\documentclass[11pt,a4paper]{article}
\usepackage[utf8]{inputenc}
\usepackage[english]{babel}
\usepackage{amsmath}
\usepackage{amsfonts}
\usepackage{amssymb}
\usepackage{amsthm}
\usepackage{graphicx}
\usepackage{epsfig}
\usepackage{cite}
\usepackage[pagewise]{lineno} 
\usepackage{dsfont}
\usepackage[all]{xy}
\usepackage{indentfirst} 
\usepackage{url}

\newtheorem{theorem}{Theorem}[section]

\newtheorem{remark}[theorem]{Remark}

\numberwithin{equation}{section}

\newcommand{\be}{\begin{eqnarray}}
\newcommand{\en}{\end{eqnarray}}
\newcommand{\no}{\nonumber}
\newcommand{\Sp}{\mathbb{S}}
\newcommand{\R}{\mathbb{R}}
\newcommand{\Hy}{\mathbb{H}}

\newcommand{\Addresses}{{
\bigskip 
\footnotesize
Instituto Federal Goiano, Campus Trindade,  Brazil \\
\textit{E-mail address:} \texttt{adriano.bezerra@ifgoiano.edu.br}
\vspace{.3cm}\\
Universidade de S\~ao Paulo, S\~ao Carlos, Brazil \\
\textit{E-mail address:} \texttt{manfio@icmc.usp.br} \\
}}

\begin{document}

\author{A. C. Bezerra and  F. Manfio}

\title{Rigidity and stability estimates for minimal submanifolds in 
the hyperbolic space}

\date{}
\maketitle

\vspace{2ex}


\begin{abstract} 
In this paper we establish conditions on the length of the second 
fundamental form of a complete minimal submanifold $M^n$ in the
hyperbolic space $\mathbb{H}^{n+m}$ in order to show that $M^n$
is totally geodesic. We also obtain sharp upper bounds 
estimates for the first eigenvalue of the super stability operator 
in the case of $M$ is a surface in $\mathbb{H}^{4}$.
\end{abstract}

\section{Introduction}

In the seminal work \cite{S2}, J. Simons established a formula
for the Laplacian of the second fundamental form of a minimal
submanifold in a space form and important applications have 
been obtained, among which we highlight one: if $M^n$ is a 
closed minimal submanifold in the unit sphere $\Sp^{n+m}$, 
with squared norm of the second fundamental form less than
$n/(2-1/m)$, then $M$ is totally geodesic. Simon's work has 
been of great interest to differential geometers, and in the last
decade several interesting gap theorems for submanifolds 
have been successfully obtained. We refer the reader to 
\cite{AM}, \cite{BRP}, \cite{BZ}, \cite{CaMa}, \cite{CIS}, 
\cite{SCK}, \cite{AB}, \cite{OL}, \cite{ENR}, \cite{NW}, 
\cite{OX}, and the references therein.

A natural problem is to ask whether a Simon's type pinching theorem 
holds for minimal submanifolds in other ambient spaces. For example,
in the case of hyperbolic space, Xia-Wang \cite{CHI} showed that the
result is true if the $L^{2}$-norm on geodesic balls of the length of the
second fundamental form of the minimal submanifold has less than 
quadratic growth and if the dimension of the submanifolds is not less
than $5$. More precisely, the condition established in \cite{CHI} is 
given by
\be\label{1.1}
\sup_{x\in M}|A|^{2}(x)<
\left\{ \begin{array}{ll}
\frac{(n+2)(n-1)^{2}}{4n}-n, & if \ m=1, \\
\frac{2}{3}\left(\frac{(mn+2)(n-1)^{2}}{4mn}-n\right), & if \ m\geq2. \\
\end{array} \right.
\en
The case of law dimension, not considered in \cite{CHI}, was studied
by Oliveira-Xia \cite{OX}. However, the condition obtained by 
Oliveira-Xia on the norm of the second fundamental form depends 
of more constants compared to condition \eqref{1.1}, and it is 
required that $n^2-6n+1+8/m>0$.

Our first main result gives an improvement of that obtained by 
Xia-Wang \cite{CHI} and Oliveira-Xia \cite{OX}. The key point is to 
make a suitable change in the condition \eqref{1.1} for 
submanifolds $M^n$ in the hyperbolic space, when $n\geq 6$. 
In fact, condition \eqref{1.3} below does not depends on the 
codimension of the submanifold, but only of the dimension.
The $L^{2}$-norm on geodesic balls of the length of the second 
fundamental form of $M$ was replaced by $L^{d}$-norm, where
$d$ was chosen in an appropriate interval. Here, $A$ and 
$B_ {p}(R)$ denote the second fundamental form of $M$ and
the geodesic ball of radius $R$ centered at $p\in M$,
respectively. 

\begin{theorem}\label{Theorem 1.1}
Let $M^n$, $n\geq 6$, be a complete immersed minimal submanifold
in the hyperbolic space $\Hy^{n+m}$. Suppose that there exists 
a constant $d\in 2(1-\sqrt{2/mn},1+\sqrt{2/mn})$ such that
\be\label{1.2}
\lim_{R\rightarrow +\infty}\frac{1}{R^{2}}\int_{B_p(R)}|A|^{d}=0.
\en
If the length of the second fundamental form $A$ of $M$ satisfies
\be\label{1.3}
\sup_{x\in M}|A|^{2}(x)<C(n):=
\left\{ \begin{array}{ll}
\frac{(n-1)^{2}}{4}-n,  & if  \ m=1,    \\
\frac{(n-1)^{2}}{6}-\frac{2}{3}n, & if \ m\geq2,  \\
\end{array} \right.
\en
then $M$ is total geodesic.
\end{theorem}

An interesting related problem is the study of the stability operator on
minimal submanifolds in the hyperbolic space $\Hy^m$. We briefly
describe now some basic facts.

Given a complete noncompact Riemannian manifold $M^n$, fix a
continuous function $\beta:M\to\R$ and consider the Laplacian
operator $\Delta$ acting on the space $C^\infty(M)$. We denote
by $\lambda_1(L_\beta,M)$ the first eigenvalue of the operator
$L_\beta=\Delta+\beta$, which is defined by
\be	\label{1.4}
\lambda_{1}(L_{\beta},M)= \inf_{f\in C_0^{\infty}(M), f\neq 0}
\frac{\int_{M}(|\nabla f|^{2}-\beta f^{2})}{\int_{M} f^{2}}.
\en
Note that, when $\beta=0$, $\lambda_1(L_0,M)$ recover the usual
first eigenvalue of $M$.

In this direction, one significant contribution is due McKean \cite{MK},
who proved  that if $M$ is simply connected  and its seccional curvature
satisfies $K_M\leq-1$, then
\be	\label{1.5}
\lambda_{1}(M)\geq\frac{(n-1)^{2}}{4} = \lambda_1(\Hy^n).
\en
In the context of submanifolds, Cheng and Leung \cite{CL} proved
that if $M^n$ is a complete minimal submanifold of $\Hy^m$, then
\be	\label{1.6}
\lambda_{1}(M)\geq\frac{(n-1)^{2}}{4}.
\en

Motivated by the second variation formula for the volume of minimal
submanifolds in the hyperbolic space, Seo \cite{KE} introduced the
concept of super stability of such submanifolds. More precisely, a
complete minimal submanifold $M^n$ of the hyperbolic space 
$\Hy^{n+m}$ is said to be {\em super stable} if
\[
\int_{M}\big(|\nabla f|^{2}-(|A|^{2}-n)f^2\big)\geq0,
\]
for all $f\in C_{0}^{\infty}(M)$. We point out that for the case of hypersurfaces,
the concept of super stability is the same as the usual definition of
stability. Recall that the stability operator of a complete minimal
hypersurface $M^n$ of $\Hy^{n+1}$ is $L_{|A|^{2}-n}$, where $A$ is
the second fundamental form of $M^n$. Moreover, it follows from
\eqref{1.4} and \eqref{1.6} that the first eigenvalue of the stability
operator of a complete totally geodesic hypersurface of $\Hy^{n+1}$ is
\be	\label{1.7}
\lambda_1(L_{|A|^{2}-n}, M)=\frac{(n-1)^{2}}{4}+n.
\en

In this direction, the first author and Wang studied in \cite{BZ} the stability
operator for complete minimal submanifolds $M^n$ in $\Hy^{n+m}$. 
They showed that if the condition \eqref{1.2} is satisfied and  the first
eigenvalue of the super stability operator is greater than a certain 
constant, then a Simon's type theorem holds if $2\leq n\leq 5$, and
for all $n\neq 3$ if $m=1$. They also obtain upper estimates for the
first eigenvalue this operator. 

\vspace{.2cm}

In our next result we present a version of \cite[Theorem 1.2]{BZ} for
minimal submanifolds $M^n$ of the hyperbolic space $\Hy^{n+m}$,
with a condition involving the norm of the second fundamental form
of $M^n$ and the first eigenvalue of the super stability operator for
dimensions $n\geq 6$. We will denote the first eigenvalue 
$\lambda_1(L_{|A|^{2}-n}, M)$ of the super stability operator of $M$
by $\overline{\lambda}_1$. 

\begin{theorem}	\label{Theorem 1.3}
Let $M^n$, $n\geq 6$, be a complete minimal submanifold
of the hyperbolic space $\Hy^{n+m}$. If the condition \eqref{1.2} is 
satisfied and
\be \label{1.10}
\sup_{x\in M}|A|^{2}(x)<2(\overline{\lambda}_{1}-2n),
\en
then $M$ is total geodesic.
\end{theorem}

Finally, the next result is an improvement of \cite[Theorem 1.1]{BZ}
by changing a pinching constant for an optimal one. More precisely,
for the case of complete minimal surfaces in $\Hy^4$, if the condition
\eqref{1.2} is satisfied for a certain constant $d$, we found an upper 
bound for $\overline{\lambda}_{1}$ by an optimal constant that 
depends only on $d$.

\begin{theorem}	\label{Theorem 1.4}
Let $M^n$, $n\geq 2$ and $n\neq 3$, be a complete minimal submanifold
of the hyperbolic space $\Hy^{n+m}$. Suppose that there exists a constant
$d$ such that \eqref{1.2} is satisfied. Then we have the following situations:
\begin{enumerate}
\item[(i)] If $m=1$, 
\be	\label{1.12}
\overline{\lambda}_{1}>\frac{n^{2}d^{2}}{4(n(d-1)+2)}+n,
\en
and
\be\label{1.13}
d\in
\left\{ \begin{array}{ll}
\left(0,1/2\right), & \mbox{if} \ n=2,  \\
\left((n-1)/n,(n-1)(n-2)/n\right), & \mbox{if} \ n=4 \ \mbox{or} \ 5, \\
(2-2\sqrt{2/n},2+2\sqrt{2/n}), & \mbox{if} \ n\geq6,  \\
\end{array} \right.
\en
then $M$ is total geodesic.
\item[(ii)] If $n=m=2$ and $d\in(2/3,2)$, then 
$\overline{\lambda}_{1}\leq\frac{d^{2}}{2d-1}+2$.
\end{enumerate}
\end{theorem}

\section{Preliminaries}

In this section we establish an inequality that will be used throughout
the paper. Given a minimal submanifold $M^n$ in the hyperbolic space 
$\mathbb{H}^{n+m}$, we have the following Kato-type inequality
\be\label{2.1}
|A|\triangle| A|+\beta(m)|A|^{4}+n| A|^{2}\geq\frac{2}{nm}|\nabla| A||^{2},
\en
with $\beta(1)=1$ and $\beta(m)=\frac{3}{2}$, if $m\geq2$ (cf. \cite{CHI}).
For every constant $\alpha>0$, and taking into account \eqref{2.1}, 
we obtain the following inequality:
\be \label{2.2}
|A|^{\alpha}\triangle|A|^{\alpha}\geq\left(1-\frac{mn-2}{mn\alpha }\right)|\nabla|A|^{\alpha}|^{2}-\beta(m)\alpha|A|^{2\alpha+2}-n\alpha |A|^{2\alpha}.
\en
Let $q$ be a nonnegative constant and let $f\in C^{\infty}_{0}(M)$. Multiplying
the inequality \eqref{2.2} for $|A|^{2\alpha q}f^{2}$ and integrating over $M$,
we have
\begin{eqnarray} \label{2.3}
\begin{aligned}
\left(1-\frac{mn-2}{mn\alpha}\right) & 
\int_{M}|\nabla|A|^{\alpha}|^{2}|A|^{2\alpha q}f^{2}
\leq\beta(m)\alpha\int_{M}|A|^{2(q+1)\alpha+2}f^{2} \\
& + 
n\alpha \int_{M}|A|^{2(q+1)\alpha}f^{2} +
\int_{M}|A|^{(2q+1)\alpha}\triangle|A|^{\alpha}f^{2}.
\end{aligned}
\end{eqnarray}
By integration by parts in the last term of \eqref{2.3}, and using the Schwarz
inequality and the Young inequality, we obtain
\begin{eqnarray} \label{2.4}
\begin{aligned}
\left(2(q+1)-\frac{mn-2}{mn\alpha}-\epsilon\right) 
\int_{M}|\nabla|A|^{\alpha}|^{2}|A|^{2\alpha q}f^{2} 
\leq \frac{1}{\epsilon}\int_{M}|A|^{2(q+1)\alpha}|\nabla f|^{2} \\
+ \beta(m)\alpha\int_{M}|A|^{2(q+1)\alpha+2}f^{2}
+n\alpha \int_{M}|A|^{2(q+1)\alpha}f^{2}.
\end{aligned}
\end{eqnarray}

\section{Proof of Theorems}

In this section we will proof the main results of our paper.

\begin{proof}[Proof of Theorem \ref{Theorem 1.1}.]
Setting $d:=2(q+1)$ and taking
$\alpha=1$, the inequality \eqref{2.4} becomes
\begin{eqnarray} \label{2.5}
\begin{aligned}
\left(d-\frac{mn-2}{mn}-\epsilon\right)&\int_{M}|\nabla|A||^{2}|A|^{2q}f^{2}
\leq \beta(m)\int_{M}|A|^{d+2}f^{2} \\
&+n\int_{M}|A|^{d}f^{2}+{1}{\epsilon}\int_{M}|A|^{d}|\nabla f|^{2}.
\end{aligned}
\end{eqnarray}
On the other hand, it follows directly from the definition of $\lambda_{1}(M)$
that
\be\label{2.6}
\int_{M}|\nabla f|^{2}\geq\lambda_{1}\int_{M}f^{2}, \ \forall \ f \in C^{\infty}_{0}(M).
\en
Plugging $f|A|^{(q+1)}$ in \eqref{2.6} and using Young's inequality, we obtain
\begin{eqnarray} \label{2.7}
\begin{aligned}
\lambda_{1}\int_{M}|A|^{d}f^{2} & \leq 
\left(1+\frac{q+1}{\epsilon}\right)\int_{M}|A|^{d}|\nabla f|^{2} \\
& +
(q+1)(q+1+\epsilon)\int_{M}|\nabla |A||^{2}|A|^{2q}f^{2}.
\end{aligned}
\end{eqnarray}
Recall that, as $M^n$ is a minimal submanifold of the hyperbolic space
$\Hy^{n+m}$, the first eigenvalue $\lambda_1$ satisfies 
$\lambda_{1}\geq\frac{(n-1)^{2}}{4}$ (cf. \cite[Corollary 3]{CL}), and
thus the inequality \eqref{2.7} becomes
\begin{eqnarray} \label{2.8}
\begin{aligned}
\frac{(n-1)^{2}}{4}\int_{M}|A|^{d}f^{2} & \leq 
\left(1+\frac{q+1}{\epsilon}\right)\int_{M}|A|^{d}|\nabla f|^{2} \\
& + 
\left(\frac{d^{2}}{4}+(q+1)\epsilon\right)\int_{M}|\nabla |A||^{2}|A|^{2q}f^{2}.
\end{aligned}
\end{eqnarray}
Since $d\in 2(1-\sqrt{2/mn},1+\sqrt{2/mn})$, we have
\be\label{2.9}
\frac{d^{2}}{4}-d+\frac{mn-2}{mn}<0,
\en
and therefore we can choose $\epsilon>0$ in such a way that, 
by using \eqref{2.5} and \eqref{2.8}, we obtain
\begin{eqnarray} \label{2.10}
\begin{aligned}
\frac{(n-1)^{2}}{4}\int_{M}|A|^{d}f^{2} & \leq 
\left(1+\frac{q+1}{\epsilon}\right)\int_{M}|A|^{d}|\nabla f|^{2}+
\beta(m)\int_{M}|A|^{d+2}f^{2} \\
& +
n\int_{M}|A|^{d}f^{2}+\frac{1}{\epsilon}\int_{M}|A|^{d}|\nabla f|^{2},
\end{aligned}
\end{eqnarray}
that is,
\be\label{2.11}
\int_{M}\left(\frac{(n-1)^{2}}{4}-n-\beta(m)|A|^{2}\right)|A|^{d}f^{2} \leq 
C \int_{M}|A|^{d}|\nabla f|^{2},
\en
where $C$ is a positive constant. Now let $f$ be a smooth function
defined on $[0,\infty)$ such that $f\geq0$, $f\equiv1$ in $[0,R]$, 
$f\equiv0$ in $[2R,+\infty)$, and with $|f'|\leq\frac{2}{R}$. Consider
the composition $f\circ r$, where $r$ is the distance function from the
point $p$. It follows from \eqref{2.11} that
\be
\no\int_{B_{p}(R)}\left(\frac{(n-1)^{2}}{4}-n-\beta(m)|A|^{2}\right)|A|^{d}\leq \frac{4C}{R^{2}}\int_{B_{p}(2R)}|A|^{d}.
\en
By letting $R\rightarrow+\infty$ and using \eqref{1.2}, we have
\be
\no\int_{M}\left(\frac{(n-1)^{2}}{4}-n-\beta(m)|A|^{2}\right)|A|^{d}\leq 0.
\en
In view of the hypothesis on $M$, given by \eqref{1.3}, we have 
\be\no
\frac{(n-1)^{2}}{4}-n-\beta(m)|A|^{2}>0,
\en
that implies that $|A|=0$, that is, $M$ is totally geodesic.
\end{proof}

\begin{proof}[Proof of Theorem \ref{Theorem 1.3}.]
It follows from the definition of $\overline{\lambda}_{1}$ that
\be\label{2.12}
\int_{M}|\nabla f|^{2}\geq\int_{M}|A|^{2}f^{2}-n\int_{M}f^{2}+\overline{\lambda}_{1}\int_{M}f^{2},
\en
for all $f\in C_0^{\infty}(M)$. Plugging $f|A|^{(q+1)}$ in the inequality
\eqref{2.12} and using Young's inequality, we obtain
\begin{eqnarray} \label{2.13}
\begin{aligned}
\int_{M}|A|^{2(q+1)+2}f^{2} & +
(\overline{\lambda}_{1}-n)\int_{M}|A|^{2(q+1)}f^{2} 
\leq \frac{\epsilon+q+1}{\epsilon}\int_{M}|A|^{2(q+1)}|\nabla f|^{2} \\
& +
(q+1)(q+1+\epsilon)\int_{M}|\nabla |A||^{2}|A|^{2q}f^{2}.
\end{aligned}
\end{eqnarray}
Under the hypothesis \eqref{1.10}, we have
\[
2n<\overline{\lambda}_{1}-\frac{1}{2}|A|^{2},
\]
what replacing in inequality (\ref{2.5}) gives us
\begin{eqnarray} \label{2.14}
\begin{aligned}
\left(d-\frac{mn-2}{mn} -\epsilon\right) & \int_{M}|\nabla|A||^{2}|A|^{2q}f^{2}
\leq \beta(m)\int_{M}|A|^{d+2}f^{2} \\
& + (\lambda_{1}-n)\int_{M}|A|^{d}f^{2} 
- \frac{1}{2}\int_{M}|A|^{d+2}f^{2} \\
& +\frac{1}{\epsilon}\int_{M}|A|^{d}|\nabla f|^{2}.
\end{aligned}
\end{eqnarray}
Regrouping the terms in \eqref{2.14}, we can see that
\begin{eqnarray} \label{2.15}
\begin{aligned}
\left(d-\frac{mn-2}{mn}-\epsilon\right) & \int_{M}|\nabla|A||^{2}|A|^{2q}f^{2}
\leq\left(\beta(m)-\frac{1}{2}\right)\int_{M}|A|^{d+2}f^{2} \\
&+(\lambda_{1}-n)\int_{M}|A|^{d}f^{2}+\frac{1}{\epsilon}\int_{M}|A|^{d}|\nabla f|^{2}.
\end{aligned}
\end{eqnarray}
Recall that $\beta(1)=1$ and $\beta(m)=\frac{3}{2}$ if $m\geq2$. Thus,
we can rewrite the inequality above as
\begin{eqnarray} \label{2.16}
\begin{aligned}
\left(d-\frac{mn-2}{mn}-\epsilon\right) & \int_{M}|\nabla|A||^{2}|A|^{2q}f^{2}
\leq \int_{M}|A|^{d+2}f^{2} \\
& + 
(\lambda_{1}-n)\int_{M}|A|^{d}f^{2}+\frac{1}{\epsilon}\int_{M}|A|^{d}|\nabla f|^{2}.
\end{aligned}
\end{eqnarray}
Setting $d=2(q+1)$ in the inequality \eqref{2.13} and relating \eqref{2.16},
we get
\begin{eqnarray*}
\begin{aligned}
\left(d-\frac{mn-2}{mn}-\epsilon\right) & \int_{M}|\nabla|A||^{2}|A|^{2q}f^{2}
\leq\left(1+\frac{q+2}{\epsilon}\right)\int_{M}|A|^{d}|\nabla f|^{2} \\
& +
\left(\frac{d^{2}}{4}+(q+1)\epsilon\right)
\int_{M}|\nabla |A|^{\alpha}|^{2}|A|^{2q\alpha}f^{2},
\end{aligned}
\end{eqnarray*}
that is
\be
\no\left(-\frac{d^{2}}{4}+d-\frac{mn-2}{mn}-(q+2)\epsilon\right)\int_{M}|\nabla|A||^{2}|A|^{2q}f^{2}\leq C\int_{M}|A|^{d}|\nabla f|^{2}.
\en
Since \eqref{2.9} also holds here, we can choose $\epsilon>0$ such that
\[
-\frac{d^{2}}{4}+d-\frac{mn-2}{mn}-(q+2)\epsilon>0.
\]
Therefore, for such $\epsilon>0$, we have
\be\label{2.17}
& &\int_{M}|\nabla|A||^{2}|A|^{2q}f^{2}\leq C\int_{M}|A|^{d}|\nabla f|^{2},
\en
where $C$ is a positive constant that depends on $d$, $m$, $n$, $q$ and 
$\epsilon$. Proceeding as in the proof of Theorem \ref{Theorem 1.1}, we 
can choose a nonnegative smooth function $f$ such that
\be \label{2.18}
\int_{B_{p}(R)}|\nabla|A||^{2}|A|^{2q}\leq \frac{4C}{R^{2}}\int_{B_{p}(2R)}|A|^{d}.
\en
Taking $R\rightarrow+\infty$ and taking into account \eqref{1.2}, we 
conclude that the second fundamental form $A$ satisfies $|A|=c=const$.
If $c\neq0$, we know from \eqref{1.2} that
\be \label{2.19}
\lim_{R\rightarrow+\infty}
\frac{Vol[B_{p}(R)]}{R^{2}}=0.
\en
It then follows from \cite{CGY} that $\lambda_{1}(M)=0$ which contradicts
with \eqref{1.6}. Therefore $|A|=0$, and this concludes the proof.
\end{proof}

\begin{proof}[Proof of Theorem \ref{Theorem 1.4}.]
By setting $d:=2(q+1)\alpha$, the inequality \eqref{2.4} becomes
\begin{eqnarray} \label{2.20}
\begin{aligned}
\frac{mn(d-1-\alpha\epsilon)+2}{mn\alpha}
& \int_{M}|\nabla|A|^{\alpha}|^{2}|A|^{2\alpha q}f^{2} \leq 
\beta\alpha\int_{M}|A|^{d+2}f^{2} \\
& +
n\alpha \int_{M}|A|^{d}f^{2} + \frac{1}{\epsilon}\int_{M}|\nabla f|^{2}|A|^{d},
\end{aligned}
\end{eqnarray}
where $\beta=\beta(m)$ is such that $\beta(1)=1$ and $\beta(m)=\frac{3}{2}$,
if $m\geq2$. Plugging $f|A|^{(q+1)\alpha}$ in \eqref{2.12} with constants
$q\geq0$ and $\alpha>0$, and using Young's inequality, we obtain
\begin{eqnarray} \label{2.21}
\begin{aligned}
\int_{M}|A|^{d+2}f^{2} & + (\overline{\lambda}_{1}-n )\int_{M}|A|^{d}f^{2} \leq
\left(1+\frac{q+1}{\epsilon}\right)\int_{M}|A|^{d}|\nabla f|^{2} \\
&+
\left(\frac{d^{2}}{4\alpha^{2}}+(q+1)\epsilon\right)\int_{M}|\nabla |A|^{\alpha}|^{2}|A|^{2q\alpha}f^{2}.
\end{aligned}
\end{eqnarray}
Multiplying \eqref{2.20} by $\frac{d^2+4(q+1)\epsilon\alpha^2}{4\alpha^2}$ 
and \eqref{2.21} by $\frac{mn(d-1-\alpha\epsilon)+2}{mn\alpha}$, and joining
these new inequalities, we get
\begin{eqnarray} \label{2.21a}
\begin{aligned}
& \frac{mn(d-1- \alpha\epsilon)+2}{mn\alpha}
\left[\int_{M}|A|^{d+2}f^{2} + (\overline{\lambda}_{1}-n)\int_{M}|A|^{d}f^{2}
\right] \\
& \leq \frac{d^2+4(q+1)\epsilon\alpha^2}{4\alpha^2}
\left[\beta\int_{M}|A|^{d+2}f^{2} + n\int_{M}|A|^{d}f^{2}\right] \\
& + C\int_{M}|A|^{d}|\nabla f|^{2},
\end{aligned}
\end{eqnarray}
where $C$ is a positive constant that depends only on $m$, $n$, $q$,
$d$, $\epsilon$ and $\alpha$. Rearranging the terms in 
\eqref{2.21a}, we obtain
\begin{eqnarray} \label{2.22}
\begin{aligned}
&\left[\frac{mn(d-1)+2}{mn\alpha}(\overline{\lambda}_{1}-n)-\frac{nd^{2}}{4\alpha}-(\overline{\lambda}_{1}-n+n\alpha(q+1))\epsilon\right]
\int_{M}|A|^{d}f^{2} \\
& +  \left[\frac{mn(d-1)+2}{mn\alpha}-\frac{\beta d^{2}}{4\alpha}-
(\beta\alpha(q+1)+1)\epsilon\right]\int_{M}|A|^{d+2}f^{2} \\
& \leq C\int_{M}|A|^{d}|\nabla f|^{2}.
\end{aligned}
\end{eqnarray}
We consider separately two possible cases: \vspace{.2cm} \\
Case $(i): m=1$. In this case, we have $\beta(1)=1$, and the 
constant that appears in the first term on the left side of 
\eqref{2.22} becomes
\be\label{2.23}
\frac{n(d-1)+2}{n\alpha}-\frac{d^{2}}{4\alpha}-(\alpha(q+1)+1)\epsilon.
\en
Moreover, by \eqref{1.13} we can see that 
$d\in2(1-\sqrt{2/n},1+\sqrt{2/n})$, for all $n\geq2$. Similarly, the 
constant that appears in the second term on the left side of 
\eqref{2.22} becomes
\be\label{2.24}
\left(\frac{n(d-1)+2}{n\alpha}\right)(\overline{\lambda}_{1}-n)-\frac{nd^{2}}{4\alpha}-(\overline{\lambda}_{1}-n+n\alpha(q+1))\epsilon.
\en
Because of \eqref{1.12}, we have
\[
\overline{\lambda}_{1}>\frac{n^{2}d^{2}}{4(n(d-1)+2)}+n.
\]
Thus, we can find $\epsilon>0$ such that \eqref{2.23} and \eqref{2.24}
are both positive. On the other hand, from \eqref{2.22}, we obtain 
\be\label{2.25}
\int_{M}|A|^{d+2}f^{2}\leq C\int_{M}|A|^{d}|\nabla f|^{2},
\en
where $C$ is a positive constante that depends only on $n$, $d$,
$q$, $\alpha$ and $\epsilon$. Using \eqref{1.2} and arguing as in 
the end of the proof of Theorem \ref{Theorem 1.1}, we conclude
that $|A|=0$ along $M$, that is, $M$ is total geodesic. Note, 
furthermore, that the hypothesis \eqref{1.12} implies that
\eqref{1.7} is satisfied. 
\vspace{.2cm} \\
Case $(ii): m=n=2.$ In this case, we have $\beta(2)=\frac{3}{2}$, 
and the constants that appears on the left side of \eqref{2.22} 
become 
\be\label{2.26}
\frac{4(d-1)+2}{4\alpha}-\frac{3d^{2}}{8\alpha}-\left(\frac{3}{2}\alpha(q+1)+1\right)\epsilon
\en
and
\be\label{2.27}
\frac{4(d-1)+2}{4\alpha}(\overline{\lambda}_{1}-2)-\frac{d^{2}}{2\alpha}
-(\overline{\lambda}_{1}-2+2\alpha(q+1))\epsilon.
\en
On the other hand, if
\be\label{2.28}
\overline{\lambda}_{1}>\frac{d^{2}}{2d-1}+2,
\en
and since $d\in(2/3,2)$, we can find $\epsilon>0$ such that the
constants \eqref{2.26} and \eqref{2.27} are both positive. Arguing
as in the end of case $(i)$, we conclude that  $|A|=0$ along $M$,
that is, $M$ is total geodesic. Now, since the first eigenvalue of the
super stability operator of a totally geodesic submanifold in 
$\Hy^{n+m}$ is given by \eqref{1.7}, we conclude in this case that
$n=2$, and thus $\overline{\lambda}_{1}=\frac{9}{4}$. Therefore, 
\eqref{2.28} makes sense if only if
\[
\frac{d^{2}}{2d-1}+2<\frac{9}{4} \ \Leftrightarrow \ 4d^{2}-2d+1<0.
\]
Since the last inequality above has negative discriminant, we have
\[
\overline{\lambda}_{1}=\lambda_{1}(L_{\mid A\mid^{2}-2}, M)
\leq\frac{d^{2}}{2d-1}+2,
\]
and this concludes the proof.
\end{proof}

\begin{remark} 
The condition \eqref{1.2} can be replaced by
\be \label{2.29}
\lim_{R\to+\infty}\frac{1}{R^{k}}\int_{B_p(R)}|A|=0.
\en
In case $(ii)$ of Theorem \ref{Theorem 1.4}, we obtain
\be\label{2.30}
\overline{\lambda}_{1}\leq\frac{(k-1)^{2}}{2k-3}+2,
\en
with $k\in(5/3,3)$ and bearing in mind that $d=1$.
From this, we can replace $k=d+1$ in \eqref{2.25} in order to get
\be \label{2.31}
\int_{M}|A|^{k+1}f^{2}\leq C\int_{M}|A|^{k-1}|\nabla f|^{2},
\en
for all $f\in C^{\infty}_0(M)$, where $C$ is a positive constant. 
Changing $f$ by $f^{\frac{k}{2}}$ in \eqref{2.31}, it follows from
H\"{o}lder inequality that
\begin{eqnarray*} 
\int_{M}|A|^{k+1}f^{k} &\leq& \overline{C}
\int_{M}|A|^{k-1}f^{k-2}|\nabla f|^{2} \\
&\leq &\overline{C}\left(\int_{M}|A|^{k+1}f^{k}\right)^{\frac{k-2}{k}}
\left(\int_{M}|A||\nabla f|^{k}\right)^{\frac{2}{k}}.
\end{eqnarray*}
Arguing as in the end of the proof of Theorem \ref{Theorem 1.1},
we can choose a nonnegative smooth function $f$ such that
\[	
\left(\int_{B_p(R)}|A|^{k+1}\right)^{\frac{2}{k}}\leq \overline{C}\left(\frac{1}{R^{k}}\int_{B_p(R)}|A|\right)^{\frac{2}{k}},
\]
and the proof follows as in the end of the proof of Theorem
\ref{Theorem 1.4}.
\end{remark}

\Addresses

\end{document}